\documentclass{article}
\usepackage[utf8]{inputenc}
\usepackage{graphicx}
\usepackage[a4paper, total={6in,8in}]{geometry}
\usepackage[export]{adjustbox}
\usepackage{subcaption}
\usepackage{float}
\usepackage{color}
\usepackage{amsmath}
\usepackage{amsthm}
\usepackage{bbm}
\usepackage{hyperref}
\usepackage[english]{babel}
\usepackage{csquotes}
\usepackage{enumitem}
\usepackage{comment}
\usepackage{fixltx2e}
\usepackage{array}
\newcolumntype{P}[1]{>{\centering\arraybackslash}p{#1}}

\usepackage{amsfonts}
\usepackage{bm}
\allowdisplaybreaks
\usepackage{algorithm2e}
\usepackage{multirow}
\usepackage{booktabs}
\usepackage{setspace}
\usepackage{adjustbox}
\usepackage{cases}

\usepackage[
backend=bibtex,
style=numeric-comp,
sorting=nyt,
giveninits=true,
citestyle=numeric-comp,
maxbibnames=99
]{biblatex}
\renewbibmacro{in:}{}
\DeclareFieldFormat[article, inbook, incollection, inproceedings, misc, thesis, unpublished]{title}{#1}
\DeclareFieldFormat[article, inbook, incollection, inproceedings, misc, thesis, unpublished]{title}{#1}

\DeclareNameAlias{sortname}{given-family}

\DeclareFieldFormat[article]{volume}{\bibstring{jourvol}\addnbspace #1}

\newcommand\norm[1]{\left\lVert#1\right\rVert}

\newtheorem{definition}{Definition}
\newtheorem{lemma}{Lemma}
\newtheorem{thm}{Theorem}

\title{\vspace{-2.5cm} \hspace{-0.125cm}A Distributionally Robust Area Under Curve Maximization Model}
\author{Wenbo Ma\footnote{Data scientist at AvalonBay Communities (AVB), {\tt wenboma2011@gmail.com}.   All views expressed in this paper are his own and do not represent the opinions of AVB.},  \ 
Miguel A. Lejeune\footnote{Corresponding Author. Department of Decision Sciences, The George Washington University, {\tt mlejeune@gwu.edu}.}
}
\date{}

\begin{document}

\maketitle

\nocite{*}
\vspace{-0.5cm}
\begin{abstract}
Area under ROC curve (AUC) is a widely used performance measure for classification models. 
We propose two new distributionally robust AUC maximization models (DR-AUC) that rely on the Kantorovich metric and approximate the AUC with the hinge loss function. We consider the two cases with respectively fixed and variable support for the worst-case distribution.
We use duality theory to reformulate the DR-AUC models and derive tractable convex optimization problems.
The numerical experiments show that the proposed DR-AUC models -- benchmarked with the standard deterministic AUC and the support vector machine models - perform better in general and in particular improve the worst-case out-of-sample performance over the majority of the considered datasets, thereby showing their robustness. 
The results are particularly encouraging since our numerical experiments are conducted with training sets of small size which have been known to be conducive to low out-of-sample performance. 
\end{abstract}

\vspace{-0.25cm}
\section{Introduction}
\vspace{-0.15cm}
Area under the receiver operating characteristic (ROC) curve (AUC) has been extensively used as a performance measure for classification models in machine learning. For example, most recent classification competitions at Kaggle \cite{KAGGLE} 
use AUC as the only evaluation metric to rank contestants and to decide prize up to hundreds of thousands of dollars. 
From a theoretical point of view, 
AUC has been showed to be statistically more consistent and more discriminating than accuracy when comparing classification models \cite{aucbetter,CortesAUC}. 
In practice, AUC is a better performance metric than accuracy when class distributions are highly imbalanced \cite{5}. 
Such a situation is common in applications, such as disease diagnosis, transaction fraud detection, and churn management, to name a few. 
\citeauthor{5} \cite{5} provide an illustration from the telecommunications industry, in which monthly churn rates are typically around 2\%. The trivial solution of labeling all customers as non-churners yields a 98\% accuracy. AUC is a performance metric that can avoid such a trivial classification since it differentiates errors made on data points with positive and negative labels, respectively (see \cite{aucbetter}). 

AUC maximization models  minimize an empirical loss function assuming that the empirical distribution is representative of the true unknown population (see, e.g., \cite{1,4,5}). As it is not always the case, in particular when the sample size is small, it is not uncommon to see AUC models with satisfactory in-sample performance have poor out-of-sample performance. Distributionally robust optimization (DRO), as a paradigm for optimization under uncertainty, has been recently presented as an effective way to remedy this issue and to improve out-of-sample performance of machine learning models (see, e.g., \cite{2,3,DROLog,droluo,chen2017robust} and the references therein). In this paper, we introduce DRO into the formulation of an AUC maximization model as an  attempt to improve our-of-sample performance.


In this study, we propose  
two new distributionally robust optimization (DRO) models for AUC maximization, in which AUC is approximated via the hinge loss function. To our knowledge, there is no published study proposing a DRO model for AUC maximization.
The optimization in the new DRO AUC maximization models are carried out over an ambiguity set of probability distributions constructed with the Kantorovich metric, also known as the order-1 Wasserstein metric. We consider the two cases with fixed and variable support of the  the worst-case distribution. We  propose a duality-based reformulation method for the min-max DRO AUC problem, which provides  computationally tractable, convex optimization reformulations. We conduct a comprehensive set of experiments to assess the out-of-sample performance of the proposed DRO AUC maximization models. The focus of the tests is the worst-case out-of-sample performance of the models when the training samples are of small size.

The remainder of this paper is organized as follows. In Section \ref{PREL}, we review the key concepts and literature related to distributionally robust optimization and AUC-based classification. 
In Section \ref{MOD}, we propose the distributionally robust AUC (DR-AUC) models and derive computationally tractable reformulations.
In Section \ref{TEST}, we test the proposed DR-AUC models on several datasets of UCI machine learning repository and benchmark the out-of-sample performance of the DR-AUC models with the deterministic AUC model and the soft-margin support vector machine.

\vspace{-0.2cm}
\section{Preliminaries} \label{PREL}
\subsection{Distributionally Robust Optimization}

Stochastic programming and robust optimization are two established frameworks for optimization under uncertainty. Distributionally robust optimization, first introduced by \citeauthor{scarf} \cite{scarf}, can be viewed as a third paradigm for optimization under certainty that seeks to bridge stochastic programming and robust optimization. 

Consider a generic loss function $L_G(w^G,x^G)$ depending on the decision variables $w^G$ and subjected to uncertainty represented by the random variables $x^G$ following a distribution $q$. 
In general, a stochastic programming model minimizes the expected value of the loss function
\begin{equation} 
\label{eqintro}
\min_{w^G} \mathbb E_{q}[L_G(w^G,x^G)] \ ,
\end{equation}
and assumes that the distribution $q$ is known. 
This assumption may not always hold in practice. On the other hand, robust optimization does not include probabilistic information and considers uncertainty via the concept of uncertainty set. The generic formulation of a robust optimization problem is:
\begin{equation} \label{eqro}
\min_{w^G} \max_{x^G \in \mathcal U} \ L_G(w^G,x^G) 
\end{equation}
where $\mathcal U$ is the uncertainty set in the uncertain coefficient space.
However, robust optimization models are often overly conservative.
Distributionally robust optimization combines ingredients from stochastic programming and robust optimization and is based on the construction of an ambiguity set in the space of probability distributions. The generic DRO model reads: 
\begin{equation} \label{eqdro}
\min_{w^G} \max_{q \in Q} \mathbb E_{q}[L_G(w^G,x^G)] 
\end{equation}
where $Q$ is the ambiguity set (see Definition \ref{def2}). Two forms of ambiguity sets prevail in the literature. 
The first type is based on moment conditions (see, e.g., \cite{delageye,3}) and assumes that the candidate distributions in the set have moments close to those of the so-called reference distribution. 
The second type is based on a probability distance function (see, e.g., \cite {2,DROLog}). In this setting, all distributions in the ambiguity set are within a prescribed distance from a reference distribution.
Several probability distance functions have been studied in literature. Among them are the Kantorovich metric (also known as degree-1 Wasserstein metric or earth mover's distance), the Kullback-Leibler (KL) divergence, 
and $\chi^2$-distance, to name a few.

In this paper, we construct a distance-based ambiguity set using the Kantorovich metric (see Definition \ref{def1}). 
Several reasons motivate this choice.  First, in a distance-based ambiguity set, modelers can control the conservatism of the problem by adjusting the size of the ambiguity set \cite{rjirrr}. Second, the optimal value of a DRO problem with a Kantorovich ambiguity set provides an upper confidence bound on the out-of-sample performance \cite{kuhnprob}. Third, the Kantorovich ambiguity set often leads to a tractable reformulation.

\vspace{-0.2cm}
\subsection{Area under ROC curve (AUC)} \label{AUCROC}

The AUC concept is closely associated with the concept of Receiver Operating Characteristics (ROC) curve, which is a two-dimensional measure of classification performance.
The AUC concept represents the area under the ROC curve. It is mathematically equivalent to the Wilconxon-Mann-Whitney statistic~\cite{5}: 
\begin{equation} \label{auc}
\text{AUC} = \frac{\sum_{i=1}^{N^+} \sum_{j=1}^{N^-} \mathbbm{1}_{f(x^{+}_{i}) \geq f(x^{-}_{j})}}{M} 
\end{equation}
where $x^+_{i}$ and $x^-_{j}$ denote data points with positive and negative labels, respectively, $f$ is a generic classification model, $\mathbbm{1}$ is an indicator function equal to 1 when $f(x^+)\geq f(x^-)$ and to  0 otherwise, $N^+$ (resp, $N^-$) is the number of data points with positive (resp., negative) labels, 
and $M=N^{+}N^{-}$ is the number of pairs of points with opposite labels ($x^{+}$,$x^{-}$).
The numerator in \eqref{auc} counts how many times the classification model $f$ assigns a larger value to a data point with a positive label than to a point with a negative label. The denominator is the number of pairs of points with opposite labels. It follows from the above definition that AUC measures the probability that a classification model assigns a larger value to a randomly selected positive data point than to a negative one \cite{rocfigure}. Its value ranges from 0 to 1. The higher the value, the better the model is. A random guess model has 0.5 as its AUC value. 

The horizontal axis in a ROC curve (Figure \ref{FIG1}) represents the  false positive rate (FPR) and the vertical axis corresponds to the true positive rate (TPR). 
The ROC curve of a classification model plots the TPR, also known as the sensitivity, against the FPR, also known as probability of false alarm or complement of the specificity, at various threshold levels. In other words, ROC describes the probability of a model correctly classifying a positive instance against incorrectly classifying a negative instance \cite{4}. 
\vspace{-0.3cm}
\begin{figure}[H] 
	\centering
	\includegraphics[scale=0.3125]{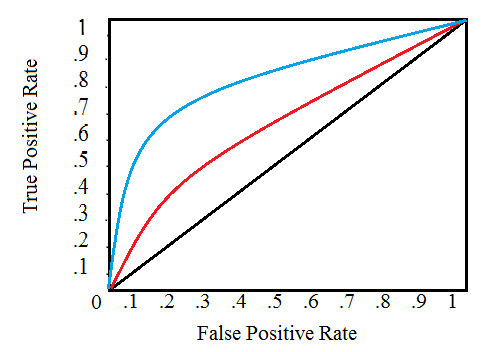}
	\vspace{-0.1in}
	\caption{ROC curve (taken from \cite{rocfigure}).
	}
	\label{FIG1}
\end{figure}
\vspace{-0.2cm}
The ROC and AUC concepts can be conveniently used to compare classification models.
Figure \ref{FIG1} displays the ROC curve for the two models A (blue line) and B (red line). Given a fixed FPR, say 0.3, model A has a higher TPR than model B and A is hence considered to be better than B at the 30\% FPR. Furthermore, it can be seen that model A generally dominates model B because A's ROC curve is above B's ROC curve for every FPR level. When no dominance relationship of one model over another can be established, one resorts to the AUC metric to compare the (classification performance of the) models.

\vspace{-0.2cm}
\subsection{\textbf{Deterministic AUC Maximization Model}}


In this section, we introduce the standard formulation of the deterministic AUC maximization model. 
In Section \ref{AUCROC}, we showed that maximizing AUC comes to maximizing the number of pairs with opposite labels in which the positive data point is attributed a larger value than the negative one. Alternatively, it is equivalent to minimizing the number of pairs with opposite labels in which the negative data point is attributed a larger value than the positive one, which is the view explicitly modelled next:
\begin{equation}
\label{aucmin}
\underset{w}{\text{min}} \
\frac{1}{M} \sum_{i=1}^{N^+} \sum_{j=1}^{N^-} \mathbbm{1}_{\{f(x^+_{i}) \leq  f(x^-_{j})\}} \ .
\end{equation}
The standard  AUC maximization model assumes that the classification function $f$ is linear, taking the form $f = w^Tx_i$, where $w^T$ is a decision vector of length $d$ and each data point  $x_i$ is characterized by a $d$-dimensional vector of features. 
 In other words, a data point $x_i$ is represented by a feature vector with each component being the feature's value. Each component $w_k$ of the decision vector $w$ multiplies the corresponding feature value $x_{ik}$ taken by point $x_i$ on feature $d$. We can rewrite \eqref{aucmin} as follows:
\begin{equation}
\label{aucsvm1.5}
\underset{w}{\text{min}} \
\frac{1}{M} \sum_{i=1}^{N^+} \sum_{j=1}^{N^-} \mathbbm{1}_{\{w^T x_{i}^+ \leq w^Tx_{j}^{-}\}} \ .
\end{equation}

Two challenges are regularly reported (see, e.g., \cite{4}) about the above model. First, the indicator function in the objective is non-differentiable and non-convex. Second, the solution to the problem may not be unique. There could be multiple optimal solutions which allow the separation of the positive and negative points with identical accuracy. To alleviate those issues, \citeauthor{4} \cite{4} proposes a constrained optimization problem obtained by introducing auxiliary non-negative decision variables $\xi_{ij}$ that permits to approximate \eqref{aucsvm1.5}: 
\vspace{-0.2cm}
\begin{subequations} \label{aucsvm2}
	\begin{align}
\underset{w, \xi}{\text{min}}
& \; \frac{1}{2} \norm{w}^{2}
+ \frac{C}{M} \sum_{i=1}^{N^+} \sum_{j=1}^{N^-} \xi_{ij} &    \label{aucsvm2-1}\\
\text{s.t} \; & \; \xi_{ij} \geq 1-(w^T x_{i}^{+} - w^T x_{j}^-) \ , &  \forall i=1,\ldots,N^+, \forall j = 1,\ldots,N^- \label{aucsvm2-2}\\
& \; \xi_{ij} \geq 0, & \forall i=1,\ldots,N^+,  \forall j = 1,\ldots,N^- \label{aucsvm2-3}
\end{align} 
\end{subequations}
Along with the minimization objective function, constraints \eqref{aucsvm2-2} and \eqref{aucsvm2-3} enforce that each variable $\xi_{ij}$ is equal to the hinge loss 
defined as $\max\{0,1-(w^Tx_i^{+}-w^Tx_j^{-})\}$ for a given $(x_i^+,x_j^-)$ pair. The hinge loss is equal to 0 if the score difference ($w^Tx_{i}^{+}-w^Tx_{j}^{-}$)  is at least 1, and is equal to 1 minus the score difference otherwise. 
The above model is called ROC optimizing SVM in \cite{4} and AUC maximizing SVM in \cite{1}. 
Problem \eqref{aucsvm2} can be viewed as a relaxed epigraphic formulation in which the nonconvex, integer-valued objective function \eqref{aucsvm1.5} is approximated by the hinge loss function and moved to the constraint set.
Problem \eqref{aucsvm2} can be equivalently reformulated as the unconstrained problem: 
\begin{equation} 
\label{aucsvm4}
\begin{aligned}
\text{\bf{D-AUC}}: \quad & \underset{w }{\text{min}}
& & \frac{1}{2} \norm{w}^{2}
+ \frac{C}{M} \sum_{i=1}^{N^+} \sum_{j=1}^{N^-} \max\{0,1-(w^T x_{i}^{+} - w^T x_{j}^-)\} 
\end{aligned}
\end{equation}
in which the $L_2$ norm term  $\norm{w}^{2}$ remedies the non-unique solution issue. 
The incorporation of the norm $\norm{w}^{2}$ breaks the ties between the (possibly) multiple optimal solutions in \eqref{aucsvm1.5}, as it leads to the selection of the model with maximal margin. We refer to  
\cite{svmsoft} for details about margin and optimal margin classification. The objective function \eqref{aucsvm1.5} is approximated by the hinge loss function $\sum_{i=1}^{N^+} \sum_{j=1}^{N^-} \max\{0,1-(w^T x_{i}^{+} - w^T x_{j}^-)\}$.   
The tuning parameter $C$ controls the trade-off between the $L_2$ norm and the loss. In the remainder of the manuscript, we refer to model \eqref{aucsvm4} as the deterministic AUC (D-AUC) maximization model, which resembles the following
support vector machine (SVM) \cite{svm-hard,svmsoft} model:
\vspace{-0.05cm}
\begin{equation} 
\label{softsvm}
\textbf{SVM}: \quad  \underset{w,b}{\text{min}} \  \frac{1}{2}\norm{w}^{2} + \frac{C}{N} \sum_{j=1}^{N} \ {\max\{0,1-y_j(w^Tx_j+b)\}} \ .
\end{equation} 
The notation $N = N^+ + N^-$ refers to the total number of data points, $y_j \in \{+1,-1\}$ is the label of each data point $j$ and $b$ is a decision variable representing the intercept of SVM's separating hyperplane and the origin. 
Although close, the AUC maximization model \eqref{aucsvm4} and the SVM model \eqref{softsvm} have significant differences. 
First, in the SVM model, the labels' values appear explicitly in the loss function $y_{j}(w^Tx_j+b)$ whereas the labels $y_j$ are not in the objective function of the D-AUC maximization model, which implicitly splits the data points into two groups (positive and negative) based on their labels and compares the value of $w^T x_{i}^{+}$ and $w^T x_{j}^-$ for each pair $(x_i^+,x_j^-)$. 
Second, for a training set of size $N= N^++N^-$, the summand operation in the objective function of the SVM model \eqref{softsvm} is carried over $N$ terms. In contrast, the summation in the D-AUC maximization model \eqref{aucsvm4} contains a (much) larger number $M =N^+ N^-$~of~terms.

\vspace{-0.2cm}
\subsection{Literature Review}

Distributionally robust machine learning models have recently been proposed in the literature. \citeauthor{3} \cite{3} propose three different distributionally robust models for ordinary least square (OLS) regression problems and use ambiguity sets based on moments and distance metrics. 
Shafieezadeh-Abadeh et al. \cite{DROLog} propose a distributionally robust logistic regression model using the Wasserstein ambiguity set. They further propose a distributionally robust approach to compute upper and lower confidence bounds on the misclassification probability. \citeauthor{2} \cite{2} study a distributionally robust framework for support vector machine using the Kantorovich distance and reformulate the model as a semi-infinite program solved with a cutting-surface algorithm. We refer the interested reader to   \citeauthor{kuhn2019wasserstein} \cite{kuhn2019wasserstein}, \citeauthor{rahimian2019distributionally} \cite{rahimian2019distributionally}, and Shafieezadeh-Abadeh et al. \cite{sorooshml2019regularization} for recent and thorough reviews on distributionally robust optimization and its applications in machine learning. 

Our work differs from the above literature in that the proposed DR-AUC maximization models use a different loss function. 
In particular, the proposed loss function  
	\eqref{aucsvm4} 	is defined on pairs of data points $(x^+_i,x^-_j)$ 
	and each term involves the comparison of each data point in the pair. This is a major difference with the loss functions in the OLS \cite{3}, SVM \cite{2} and logistics regression \cite{DROLog} studies above-mentioned, in which each term of the loss function corresponds to a point and not - as here - to a pair of points with different labels.
In other words, the proposed DR-AUC models optimize the pairwise loss between two data points from opposite classes while the loss functions in the DRO machine learning literature calculate the sum of the losses over each individual data point.

The mathematical expression of AUC can be written in analytical form and is -- as above-mentioned  -- integer-valued, non-differentiable, and non-convex.
That is why many studies optimize an approximate function of AUC that is easier to solve numerically. \citeauthor{5} \cite{5} propose a truncated quadratic function to approximate AUC and use gradient based methods to train the model.  \citeauthor{1} \cite{1} and \citeauthor{4} \cite{4} use  a hinge loss function to derive an approximation in their AUC models. Similar to SVM models, they also add a $L_{2}$ regularization term in the objective function. Recently  \citeauthor{Norton2018} \cite{Norton2018} propose a novel metric called Buffered AUC (bAUC). They show that bAUC is the tightest concave lower bound of AUC and that the optimization of bAUC can often be reformulated as a convex or even linear problem. It is important to note that all the above models in \cite{1,4,5} are deterministic optimization models without explicit consideration for data uncertainty. 
This is a significant difference with the proposed DR-AUC models that explicitly account for data uncertainty, more precisely uncertainty related to the features of the data points, by using a DRO approach. Not only do we assume that the features characterizing the data points are uncertain, we also consider that their probability distribution is only known imperfectly. 

\vspace{-0.25cm}
\section{DR-AUC Models and Reformulations} \label{MOD}

According to model (\ref{aucsvm4}), the D-AUC maximization model is:
\begin{equation} \label{aucsvmsec3}
\begin{aligned}
& \underset{w}{\text{min}}
& & \frac{1}{2} \norm{w}^{2}
+ \frac{C}{M} \sum_{i=1}^{N^+} \sum_{j=1}^{N^-} \max\{0,1-(w^T x_{i}^{+} - w^T x_{j}^-)\} \ .
\end{aligned}
\end{equation}

The second term $\frac{1}{M}\sum_{i}\sum_{j}\max\{0,1-(w^T x_{i}^{+} - w^T x_{j}^-)\}$ in \eqref{aucsvmsec3} can be interpreted as the empirical risk assuming that the realized sample follows a uniform distribution in which each atom, represented by the tuple $(x^+,x^-)$, has a probability weight $\frac{1}{M}$. If  $\hat{p}$ refers to the empirical distribution and $h(w;x^+,x^-) = \max\{0,1-(w^T x_{i}^{+} - w^T x_{j}^-)\}$, model \eqref{aucsvmsec3} can be rewritten as 
\begin{equation} \label{empirical loss}
\underset{w}{\text{min}} \ \frac{1}{2}\norm{w}^{2} + C E_{\hat{p}}[h(w;x^{+},x^{-})]
\end{equation}
The primary goal of a classification model is to minimize its expected misclassification risk on unseen data. One of the assumptions of model \eqref{empirical loss} is that the unseen data follow the uniform distribution described above. It is however rarely the case, which can lead to a high classification performance on the training set and to a much weaker prediction performance, i.e., performance on unseen data in the test set.  
To relax the assumption that $\hat{p}$ defines the unknown population, we consider a set  $\mathcal{P}$ of distributions, called ambiguity set, and the optimization (and classification) is carried over the worst-case distribution $p^*$ in $\mathcal{P}$ that maximizes the expected loss $E_{p}[h(w;x^{+},x^{-})]$:
\begin{equation} \label{eq:1}
\underset{w}{\text{min}} \ \frac{1}{2}\norm{w}^{2} + C \sup_{p \in \mathcal{P} }E_{p}[h(w;x^{+},x^{-})]    \ .
\end{equation}
We call model \eqref{eq:1} distributionally robust AUC (DR-AUC), since this min-max problem assumes that the probability distribution of the tuple $z=(x^+,x^-)$, hereafter referred to as an atom, is imperfectly known and aims to hedge against the worst-case distribution in the set $\mathcal{P}$ of plausible distributions.  
The ambiguity set $\mathcal{P}$ in \eqref{eq:1} is a set of unknown probability distributions defined using the Kantorovich metric, which we define now along with the Kantorovich metric-based ambiguity set.

\begin{definition} {\label{def1}}
    (Kantorovich metric) The Kantorovich metric is a distance function between two probability distributions $p_1$ and $p_2$  
    \begin{subequations} \label{kmetric}
    	\begin{align}
        d_p(p_1,p_2) =  \inf_{K}  &\int_{Z_1 \times Z_2} \norm{z_1-z_2}K(dz_1,dz_2) \label{kmetric-1}\\ 
       \text{s.t} & \int_{Z_2}K(z_1,dz_2) = p_1 \label{kmetric-2}\\
       & \int_{Z_1}K(dz_1,z_2) = p_2 \label{kmetric-3} 
    	\end{align} 
    \end{subequations}
where $Z_1$ and $Z_2$ are the supports of $p_1$ and $p_2$,  $z_1$ and $z_2$ are the atoms of $Z_1$ and $Z_2$, $\norm{\cdot}$ is a norm, and $K$ is a
 joint density function defined over pairs of atoms ($z_1,z_2$). 
\end{definition}
The objective function \eqref{kmetric-1} minimizes the total cost of transporting probability mass between $p_1$ and $p_2$ and $\norm{z_1-z_2}K(dz_1,dz_2)$ is the unit transportation cost between $p_1$ and $p_2$. 
The constraints \eqref{kmetric-2} and \eqref{kmetric-3} enforce the requirements on the marginal distributions $p_1$ and $p_2$.
The Kantorovich metric can be interpreted as the minimum cost for transporting probability mass between $p_1$ and $p_2$ and the optimal joint density function $K$ defines the corresponding transportation plan.

\begin{definition} \label{def2}
    (Kantorovich ambiguity set) The Kantorovich ambiguity set is a set of probability distributions $p$ that are within a ball of radius $\epsilon$ from a reference distribution $\hat{p}$:
    \begin{equation} \label{set}
    \begin{aligned}
        \mathcal{P} = \{p \ | \ d_p(p,\hat{p}) \leq \epsilon \}
    \end{aligned}
    \end{equation}
    where $d_p(\cdot,\cdot)$ refers to the Kantorovich metric 
    and $\epsilon$ is the radius of the Kantorovich ambiguity set.
\end{definition}
\noindent
The parameter $\epsilon$ can be used to control the size of the ambiguity set and the conservatism level; $\epsilon$ is also related to the probabilistic guarantee that the ambiguity set contains the true distribution (see \cite{kuhnprob}).

In this paper, we consider two cases for the worst-case distribution. First, we assume that any distribution in the ambiguity set $\mathcal{P}$ has the same finite number of atoms as the reference one and that the locations of these atoms are identical to the ones in the reference distribution. The worse-case distribution is found by changing the probability mass associated with each atom. We call the resulting model DR-AUC with fixed support (DR-AUC-F). Second, we assume that neither the number of atoms nor their locations are known. We name the corresponding model DR-AUC-V with variable support. In what follows we provide the reformulation for each case. 

\subsection{DR-AUC Model with Fixed Support}
\label{DR-AUC-Fsection}
In the fixed support case, we use the empirical uniform distribution with probability $\frac{1}{M}$ as the reference distribution. 
We denote by $\hat{p}$ the empirical distribution with each atom having probability mass $\frac{1}{M}$, by $p$ any distribution in 
 $\mathcal{P}$, and by $p^*$ the worst-case distribution in $\mathcal{P}$. Since $p$  and $\hat{p}$ are assumed to have the same known atoms, we denote $\hat{Z}$ as the support and $\hat{z}=(\hat{x}^{+},\hat{x}^{-})$ as the atoms of both $p$  and $\hat{p}$. 

%


%
The DR-AUC-F model reads now: 
\begin{equation}\label{case1t}
\textbf{DR-AUC-F}: \quad \underset{w}{\text{min}} \ \big[ \frac{1}{2}\norm{w}^{2} + C\max_{p \in \mathcal{P} } \sum_{\hat{z} \in \hat{Z}}h(w;\hat{z})p(\hat{z}) \big]
\vspace{-0.24cm}
\end{equation}
\vspace{-0.14cm}
and the ambiguity set $\mathcal{P}$ is defined as:
\begin{equation}\label{case1}
\begin{aligned}
\mathcal{P} = \{p \geq 0 \ | \ \exists K \geq 0 \ \text{s.t.} 
&\sum_{i=1}^{M} K(\hat{z}_{i},\hat{z}_{j}) = p(\hat{z}_{j}), \quad \forall j \in \{1,\ldots,M\} \\
&\sum_{j=1}^{M} K(\hat{z}_{i},\hat{z}_{j}) = \frac{1}{M}, \quad \forall i \in \{1,\ldots,M\} \\
&\sum_{i=1}^{M} \sum_{j=1}^{M} d_z(\hat{z}_{i},\hat{z}_{j})K(\hat{z}_{i},\hat{z}_{j}) \leq \epsilon \ 
\} \ .
\end{aligned}
\end{equation}
The notation $p(\hat{z}_{j})$ refers to the unknown probability of the $j^{th}$ atom in $p$ and 
$d_z(\hat{z}_i,\hat{z}_j) = \norm{\hat{x}^+_{i}-\hat{x}^+_{j}}_{1} + \norm{\hat{x}^-_{i}-\hat{x}^-_{j}}_{1}$ represents the distance between $\hat{z}_i$ and $\hat{z}_{j}$. 
To simplify the notation, we use thereafter  $K_{ij}$ instead of $K(\hat{z}_{i},\hat{z}_{j})$, $p_{j}$  instead of $p(\hat{z}_{j})$, and $h_j$  instead of
$h(w;\hat{z}_{j})$.
Inserting the representation \eqref{case1} of the ambiguity set $\mathcal{P}$  into \eqref{case1t}, the inner maximization problem in \eqref{case1t} takes the form of the following linear programming problem
\begin{subequations}\label{innerprimal}
\begin{align}
\label{innerprimalobj}
\underset{K_{ij} \geq 0}{\text{max}} & \quad \sum_{j=1}^{M} h_{j} \sum_{i=1}^{M} K_{ij} & \\
\label{case1eqct1}
\text{s.t.} & \quad \sum_{j=1}^{M} K_{ij} = \frac{1}{M}  & \forall i \in \{1,\ldots,M\} \\
\label{case1ieqct3}
& \quad \sum_{i=1}^{M} \sum_{j=1}^{M} d_z(\hat{z}_{i},\hat{z}_{j})K_{ij} \leq \epsilon  &
\end{align}
\end{subequations}
with decision variables $K_{ij}$.

\begin{thm}
	\label{THM1}
Let $t_i$ be the sign-unrestricted Lagrangian multipliers associated to the equality constraints \eqref{case1eqct1} and let $\lambda$ be the non-negative Lagrangian multiplier associated to the inequality constraint\eqref{case1ieqct3}.
	The dual of $\eqref{innerprimal}$ is the linear problem:
	\vspace{-0.1cm}
\begin{subequations} \label{eq6}
\begin{align}
\underset{\lambda \geq 0, t_i}{\text{min}}
& \quad \sum_{i=1}^{M} t_{i}\frac{1}{M} 
+ \lambda \epsilon & \\
\text{s.t.} \ & \quad t_i + \lambda d_z(\hat{z}_{i},\hat{z}_{j}) \geq h_{j}  & \forall i \in \{1,\ldots,M\}, \forall j \in \{1,\ldots,M\}. 
\end{align}
\end{subequations}
\end{thm}
\begin{proof}

The Lagrangian of problem (\ref{innerprimal}) is
\vspace{-0.1cm}
\begin{align}
L & = \sum_{j=1}^{M} h_{j}\sum_{i=1}^{M} K_{ij} + 
\sum_{i=1}^{M}t_{i}(\frac{1}{M}-\sum_{j=1}^{M} K_{ij} ) + \lambda(\epsilon-\sum_{i=1}^{M} \sum_{j=1}^{M} d_z(\hat{z}_{i},\hat{z}_{j})K_{ij}) \\
 & = 
 -\sum_{i=1}^{M} \sum_{j=1}^{M}K_{ij}(t_i + \lambda  d_z(\hat{z}_{i},\hat{z}_{j}) - h_j)
 +\sum_{i=1}^{M}t_{i}\frac{1}{M}  +\lambda \epsilon \label{lag1}
\end{align}
and the corresponding Lagrangian relaxation problem is
	\begin{align}
	\label{l_relax}
	\underset{K_{ij} \geq 0}{\text{max}} & L \ .
	\end{align}
	
\vspace{-0.122cm}

Since \eqref{innerprimal} is a linear problem and hence enjoys strong duality, the optimal value of its Lagrangian relaxation problem  \eqref{l_relax} is a finite upper bound of the optimal value of the primal \eqref{innerprimal}. The finite property holds only when the following conditions are satisfied:
\begin{equation} \label{eq:9}
t_i + \lambda d_z(\hat{z}_{i},\hat{z}_{j}) - h_j \geq 0, \forall i \in \{1,\ldots,M\}, \quad \forall j \in \{1,\ldots,M\}
\end{equation}

\vspace{-0.1cm}
Combining \eqref{l_relax} and \eqref{eq:9}, we obtain a new relaxation problem 
\begin{subequations} \label{primalint}
\begin{align}
\max_{K_{ij} \geq 0} & \quad 
-\sum_{i=1}^{M} \sum_{j=1}^{M}K_{ij}(\mu_j + t_i + \lambda d_z(\hat{z}_{i},\hat{z}_{j}))
+\sum_{i=1}^{M}t_{i}\frac{1}{M}  +\lambda \epsilon & \\
\text{s.t.} \ & \quad t_i + \lambda d_z(\hat{z}_{i},\hat{z}_{j}) - h_{j} \geq 0 & \forall i \in \{1,\ldots,M\}, \ \forall j \in \{1,\ldots,M\}
\end{align}
\end{subequations}
whose optimal value is a finite upper bound for (\ref{innerprimal}).
It is straightforward to see that an optimal solution for (\ref{primalint}) is obtained by setting $K_{ij}=0$. Therefore, the dual problem of \eqref{innerprimal} is:
\vspace{-0.12cm}
\begin{subequations} \label{eq10}
\begin{align}
\underset{\lambda \geq 0 ,t_i}{\text{min}}
& \ \sum_{i=1}^{M} t_{i}\frac{1}{M}
 +\lambda \epsilon & \\
\label{eq10:1}
\text{s.t.} \ & \ t_i + \lambda d_z(\hat{z}_{i},\hat{z}_{j}) - h_{j} \geq   0 & \quad \forall i \in \{1,\ldots,M\}, \ \forall j \in \{1,\ldots,M\}
\end{align}
\vspace{-0.1cm}
\end{subequations}
\end{proof}
\noindent
We can now provide a convex quadratic reformulation of the DR-AUC-F model.
\begin{lemma}
The DR-AUC-F model \eqref{case1t} can be equivalently reformulated as the following convex quadratic programming problem:
\vspace{-0.1cm}
\begin{equation} \label{case1final}
	\begin{aligned}
	\underset{w,\lambda \geq 0 ,t_i}{\text{min}}
	&   \frac{1}{2}\norm{w}^{2} + C(\sum_{i=1}^{M} t_{i}\frac{1}{M} 
	+ \lambda \epsilon) & \\
	\text{s.t.}
	\ &  t_i + \lambda (\norm{\hat{x}^+_{i}-\hat{x}^+_{j}}_{1} + \norm{\hat{x}^-_{i}-\hat{x}^-_{j}}_{1}) \geq 1-w^T (\hat{x}^{+}_j -\hat{x}_{j}^-) & \ 
	\forall i \in \{1,\ldots,M\}, \quad \forall j \in \{1,\ldots,M\} \\
	&  t_i + \lambda (\norm{\hat{x}^+_{i}-\hat{x}^+_{j}}_{1} + \norm{\hat{x}^-_{i}-\hat{x}^-_{j}}_{1}) \geq 0 & \ \forall i \in \{1,\ldots,M\}, \forall j \in \{1,\ldots,M\}
	\end{aligned}	
\end{equation}
\end{lemma}
\begin{proof}
Replacing the inner maximization problem in \eqref{case1t} by its dual representation \eqref{eq6} (Theorem \ref{THM1}), linearizing the expression $h=\max\{0,1-(w^T \hat{x}^{+} - w^T \hat{x}^-)\}$, and setting $d_z(\hat{z}_i,\hat{z}_j) = \norm{\hat{x}^+_{i}-\hat{x}^+_{j}}_{1} + \norm{\hat{x}^-_{i}-\hat{x}^-_{j}}_{1}$ gives the formulation  \eqref{case1final} and provides the result we set out to prove.
\end{proof}
The size of the reformulated DR-AUC problem \eqref{case1final} is closely related to the size of the training set. For a dataset with $N^+$ positive data points and $N^-$ negative ones, thereby giving a number $M=N^+N^-$ of atoms in the empirical distribution, and a feature vector of size $d$, 
the DR-AUC-F model \eqref{case1final} has $(d+M+1)$ decision variables and ($2M^2$+1) constraints.

\subsection{DR-AUC Model with Variable Support}
\label{DR-AUC-Vsection}
In the variable support case, we assume that neither the number of atoms in the worst-case distribution nor their locations are known. We use the result presented in \cite{2} (see proof of Theorem 2 therein) and extend it here to the DR-AUC model. The corresponding optimization model DR-AUC-V reads:

\begin{equation}\label{case2t}
\textbf{DR-AUC-V}: \quad \underset{w}{\text{min}} \ \big[ \frac{1}{2}\norm{w}^{2} + C\sup_{p \in \mathcal{P} } \int_{Z}h(w;z)p(dz) \big]
\vspace{-0.1cm}
\end{equation}
\vspace{-0.14cm}
and the corresponding ambiguity set $\mathcal{P}$ is defined as:
\begin{equation}\label{case2}
\begin{aligned}
\mathcal{P} = \{p \geq 0 \ | \ \exists K \geq 0 \ \text{s.t.} 
&\sum_{i=1}^{M} K(\hat{z}_{i}, z) = p(z), \quad \forall z \in Z \\
&\int_{Z} K(\hat{z}_{i},dz) = \frac{1}{M}, \quad \forall i \in \{1,\ldots,M\} \\
&  \sum_{i=1}^{M} \int_{Z} d_z(\hat{z}_{i}, z)K(\hat{z}_{i}, dz) \leq \epsilon \} \ ,
\end{aligned}
\end{equation}
where $Z$ is the unknown support of the worst-case distribution.

\begin{thm}
	\label{THM2}
	The DR-AUC-V model is equivalent to the following convex optimization problem:
	\vspace{-0.021in}
\begin{subequations}\label{case2final}
	\begin{align}
	\label{case2final1}
	\underset{w,\lambda \geq 0 ,t_i}{\text{min}} \; &   \frac{1}{2}\norm{w}^{2} + C\left(\frac{1}{M}\sum_{i=1}^{M} t_{i} + \lambda \epsilon \right) & \\
	\label{case2final2}
	\text{s.t.} \; & t_i  \geq 1-w^T (\hat{x}^{+}_i -\hat{x}_{i}^-)  & 	\forall i \in \{1,\ldots,M\} \\
	\label{case2final3}
	& t_i  \geq 0 &  	\forall i \in \{1,\ldots,M\} \\
	\label{case2final4}
	&  \norm{w}_{\infty} \leq \lambda &
	\end{align}
\end{subequations}
	where $\norm{}_{\infty}$ denotes the infinity norm.
\end{thm}
\begin{proof}
	Using Lemma 45 in \cite{sorooshml2019regularization}, we can equivalently rewrite the inner maximization problem in \eqref{case2t} as:
\begin{equation} \label{inner_case2_step1}
\begin{aligned}
\underset{\lambda \geq 0 ,t_i}{\text{min}}
&   \; \frac{1}{M} \sum_{i=1}^{M} t_{i}
+ \lambda \epsilon & \\
\text{s.t.}
\ &  \; t_i  \geq \underset{z = (x^+,x^-) \in Z}{\text{sup}} h(z)-\lambda d_z(z_i, z) & \ \quad \forall i \in \{1,\ldots,M\} 
\end{aligned}.
\end{equation}
The direct application of Lemma 47 in \cite{sorooshml2019regularization} gives the following equivalent reformulation of \eqref{inner_case2_step1}: 
\begin{equation} \label{inner_case2_step2}
\begin{aligned}
\underset{\lambda \geq 0 ,t_i}{\text{min}}
&  \; \frac{1}{M} \sum_{i=1}^{M} t_{i}
+ \lambda \epsilon & \\
\text{s.t.}
\ &  \; t_i  \geq h(z_i) &  \forall i \in \{1,\ldots,M\} \\
\ & \; \text{lip(h)}\norm{w}_{\infty} \leq \lambda &
\end{aligned} \; .
\end{equation}
Since the Lipschitz constant of the hinge loss function lip(h) is $1$, the last constraint in \eqref{inner_case2_step2} reduces to  $\norm{w}_{\infty} \leq \lambda$. Incorporating the outer minimization problem in \eqref{case2t} and linearizing $h(z_i) = \text{max}\{0, 1 - (w^Tx_{i}^+ - w^Tx_{i}^-)\}$, 
we obtain \eqref{case2final}.
\end{proof}

\vspace{-0.285cm}
\section{Numerical Experiments} \label{TEST}
The goal of this section is to assess the out-of-sample performance of the proposed DR-AUC models  DR-AUC-F and DR-AUC-V, particularly when the training set is small, which has been known to be conducive to weak out-of-sample performance. 
We conduct the numerical experiments using five publicly available datasets (Table \ref{T1}) from the UCI machine learning repository \cite{uclmlrepo}. 
We benchmark the out-of-sample classification  performance and robustness of the proposed DR-AUC models with those of the deterministic AUC model (D-AUC) \eqref{aucsvm2} and the support vector machine (SVM) model \eqref{softsvm} as SVM is one of the most widely used off-the-shelf models, has some similarity with AUC models, and has been shown to be equivalent to maximizing an approximation of AUC when the data are linearly separable \cite{4}.


\begin{center}
\begin{tabular}{ |c||c|c|c|  }
 \hline
 Dataset & Number of Features & Sample Size & Application Area\\
 \hline
 Banknote Authentication (BA)   & 4    & 1372 & Finance\\
 Vertebral Column (VC)&   6  & 310   & Healthcare\\
 Pima Indians Diabetes (PID) & 8 & 768 &  Healthcare\\
 Ionosphere (Ino) & 34 & 351&  Aerospace\\
 Statlog Heart (SH) &  13  & 270 & Healthcare\\
 \hline
\end{tabular}
\vspace{-0.2cm}
\captionof{table}{Decsription of Datasets}
	\label{T1}
\end{center}

 Our test environment is Python 3.7 and AMPL 3.5.0 (CPLEX 12.9.0.0) in a Windows 10 environment. We use a grid search and a 5-fold cross-validation approach to select values for the tuning parameter $C$ in the DR-AUC, D-AUC, and SVM models and for the $\epsilon$ parameter in DR-AUC model. The candidate values used in the grid search are $[0.0001, 0.001, 0.01, 0.1, 1, 5, 10, 50]$ for $C$ in SVM and D-AUC, $[0.1, 1, 2.5, 5, 10]$ for $C$ in DR-AUC and $[0.01, 0.1, 0.5, 1, 5, 10]$ for $\epsilon$. Specifically, we split each dataset into five subsets, use any four of them to train the model and calculate the AUC value \eqref{auc} using the remaining one. We select the parameter values with the highest mean AUC 
 as the ones used in our numerical study.
 
For each dataset in Table \ref{T1}, we run 100 experiments. In every experiment, we use stratified sampling to select 60 data points as the training set and build the classification model on it.
We only select 60 points in the training set since our primary objective is to assess the out-of-sample performance of the DR-AUC models with small, and prone to out-of-sample mistakes, datasets.
Next, we apply the model to the rest of the data points and calculate the AUC value \eqref{auc} on those to obtain the out-of-sample performance.

For each model, we calculate the mean AUC over the worst 10 experiments to evaluate the model's worst-case classification performance (the higher, the better) which is indicative of the robustness of the models. 
We use the relative difference (R. Diff.) (as in \cite{droluo}) of mean AUC defined as (AUC\textsubscript{DR-AUC}-AUC\textsubscript{benchmark})/(1 - AUC\textsubscript{benchmark}) to compare the models' worst-case performance. The reason for using the relative difference is that when the benchmark model has a high AUC score (AUC \textsubscript{benchmark}), the margin for improvement is small \cite{droluo}.  
 The worst-case performance of the DR-AUC-F and DR-AUC-V models are displayed below in Table \ref{T2} and \ref{T3}.
\begin{center}
	\begin{tabular}{|c||P{2.5cm}|P{2.5cm}|P{2.6cm}|P{2.5cm}|P{2.9cm}|}
		\hline
		Dataset& SVM & D-AUC  &  \textbf{DR-AUC-F}  & R. Diff. vs SVM & R. Diff. vs D-AUC \\
		\hline
		BA & 0.9978 $\pm$ 0.0004 & 0.9979 $\pm$ 0.0004 & 0.9979 $\pm$ 0.0004  & 4.55\% & 0.00\% \\
		VC & 0.8695 $\pm$ 0.0131  & 0.8819 $\pm$ 0.0154 & 0.8818 $\pm$ 0.0158  & 9.43\% & -0.08\% \\
		PID & 0.7256 $\pm$ 0.0301 & 0.7065 $\pm$ 0.0233 & 0.7353 $\pm$ 0.0309  & 3.53\% &  9.81\% \\
		Ion & 0.7939 $\pm$ 0.0282  & 0.7974 $\pm$ 0.0179 &  0.8002 $\pm$ 0.0238 & 3.06\% & 1.38\%  \\
		SH & 0.7816 $\pm$ 0.0302  & 0.8218 $\pm$ 0.0081 & 0.8333 $\pm$ 0.0093 &  23.67\% & 6.45\%  \\
		\hline
	\end{tabular}
	\captionof{table}{Worst-Case Performance for DR-AUC-F: AUC Average and Standard Deviation for 10 Worst Experiments}
	\label{T2}
\end{center}
Table \ref{T2} shows that the DR-AUC model DR-AUC-F with fixed support \eqref{case1t} presented in Section \ref{DR-AUC-Fsection} outperforms the SVM model on all datasets in terms of the worst-case performance. Compared with the deterministic AUC model D-AUC, DR-AUC-F achieves a significantly better worst-case performance on three of the five datasets (i.e., PID, Ion, SH). The two models perform similarly on the other two. 

\vspace{-0.1in}
\begin{center}
	\begin{tabular}{|c||P{2.5cm}|P{2.5cm}|P{2.6cm}|P{2.5cm}|P{2.9cm}|}
		\hline
		Dataset& SVM & D-AUC  &  \textbf{DR-AUC-V}  & R. Diff. vs SVM & R. Diff. vs D-AUC \\
		\hline
		BA & 0.9978 $\pm$ 0.0004 & 0.9979 $\pm$ 0.0004 & 0.9980 $\pm$ 0.0005  & 9.09\% & 4.76\% \\
		VC & 0.8695 $\pm$ 0.0131  & 0.8819 $\pm$ 0.0154 & 0.8813 $\pm$ 0.0159  & 9.04\% & -0.51\% \\
		PID & 0.7256 $\pm$ 0.0301 & 0.7065 $\pm$ 0.0233 & 0.7276 $\pm$ 0.0383  & 0.73\% &  7.19\% \\
		Ion & 0.7939 $\pm$ 0.0282  & 0.7974 $\pm$ 0.0179 &  0.8057 $\pm$ 0.0117 & 5.72\% & 4.10\%  \\
		SH & 0.7816 $\pm$ 0.0302  & 0.8218 $\pm$ 0.0081 & 0.8345 $\pm$ 0.0097 &  24.22\% & 7.13\%  \\
		\hline
	\end{tabular}
	\captionof{table}{Worst-Case Performance for DR-AUC-V: AUC Average and Standard Deviation for 10 Worst Experiments.}
	\label{T3}
\end{center}
We carry out the same analysis for the DR-AUC model DR-AUC-V with variable support \eqref{case2t} presented in  Section \ref{DR-AUC-Vsection}. Table \ref{T3} shows that DR-AUC-V outperforms SVM on each dataset and is better than D-AUC on four of the five datasets. The D-AUC model is marginally better than DR-AUC-V on the remaining dataset. 

While Tables \ref{T2} and \ref{T3} show that the proposed models DR-AUC-F and DR-AUC-V generally outperform the two benchmark models D-AUC and SVM, there is no clear indication of a dominance relationship between DR-AUC-F and DR-AUC-V. The DR-AUC-V model is better than DR-AUC-F for three datasets and the reverse is true for the other two datasets (i.e., compare the fourth column of Tables \ref{T2} and \ref{T3}).

\begin{center}
	\begin{tabular}{|c||P{2.7cm}|P{2.7cm}|P{2.7cm}|P{2.7cm}|}
		\hline
		Dataset & SVM & D-AUC  &  DR-AUC-F & DR-AUC-V  \\
		\hline
		BA & 0.9991 $\pm$ 0.0006 & 0.9992 $\pm$ 0.0006 & 0.9992 $\pm$ 0.0005 & 0.9992 $\pm$ 0.0006  \\
		VC & 0.9145 $\pm$ 0.0194  & 0.9191 $\pm$ 0.0164 &  0.9191 $\pm$ 0.0165 & 0.9191 $\pm$ 0.0166 \\
		PID & 0.7856 $\pm$ 0.0262 & 0.7852 $\pm$ 0.0328 &  0.7907 $\pm$ 0.0254 & 0.7880 $\pm$ 0.0282   \\
		Ion & 0.8511 $\pm$ 0.0230  & 0.8525 $\pm$ 0.0286 & 0.8555 $\pm$ 0.0291 & 0.8583 $\pm$ 0.0272    \\
		SH & 0.8528 $\pm$ 0.0338  & 0.8698 $\pm$ 0.0230 & 0.8749 $\pm$ 0.0194 & 0.8753 $\pm$ 0.0195   \\
		\hline
	\end{tabular}
	\captionof{table}{Overall Performance: The Average and Standard Deviation of AUC among the 100 Experiments }
	\label{T4}
\end{center}
In Table \ref{T4}, we analyze the overall performance among all the 100 experiments for each model. We see that both DR-AUC-F and DR-AUC-V outperform the benchmark models SVM and D-AUC on three datasets (i.e., PID, Ion and SH). The two DR-AUC models tie with D-AUC and slightly perform better than SVM on the other two datasets BA and VC. The differences are  less obvious than for the worst-case analysis presented in Tables \ref{T2} and \ref{T3}. The two proposed models DR-AUC-F and DR-AUC-V have extremely similar performance profiles (see columns 4 and 5 in Table \ref{T4}).

	In summary, the result above demonstrates that the proposed DR-AUC models are more robust and have a better worst-case performance than the SVM and deterministic AUC models. 
\vspace{0.3cm}	

	\textbf{Acknowledgments}: Miguel Lejeune was partially supported by the Office of Naval Research, Grant \#N000141712420.

\printbibliography

\end{document}